\documentclass[12pt]{amsart}

\usepackage{pb-diagram}

\usepackage{amsfonts}
\usepackage{amsmath}
\usepackage{amssymb}

\allowdisplaybreaks

\newcommand{\dbar}{\overline{\partial}}

\newcommand{\ddbar}{\sqrt{-1}\partial\dbar}

\newtheorem{theorem}{Theorem}[section]
\newtheorem{proposition}{Proposition}[section]
\newtheorem{lemma}{Lemma}[section]

\newtheorem{definition}{Definition}[section]

\renewcommand{\thefootnote}{\fnsymbol{footnote}}
\newcommand{\starttext}{ \setcounter{footnote}{0}
\renewcommand{\thefootnote}{\arabic{footnote}}}

\newcommand{\beq}{\begin{equation}}
\newcommand{\bea}{\begin{eqnarray}}
\newcommand{\eea}{\end{eqnarray}} \newcommand{\ee}{\end{equation}}

\def\ba{\begin{eqnarray}}
\def\ea{\end{eqnarray}}

\def\ti\tilde

\def\u{\underline}

\def\ti{\tilde}

\begin{document}

\starttext \baselineskip=18pt \setcounter{footnote}{0}

\title[Diameter estimates in K\"ahler geometry II]{Diameter estimates in K\"ahler geometry II: \\
{\small removing the small degeneracy assumption} }

\author[{Bin Guo, Duong H. Phong, Jian Song and Jacob Sturm}
]{Bin Guo$^*$, Duong H. Phong$^\dagger$, Jian Song$^\ddagger$ and Jacob Sturm$^{\dagger\dagger}$ }

\thanks{Work supported in part by the National Science Foundation under grants DMS-22-03273, DMS-22-03607 and DMS-23-03508, and the collaboration grant 946730 from Simons Foundation.}

\address{$^*$ Department of Mathematics \& Computer Science, Rutgers University, Newark, NJ 07102}

\email{bguo@rutgers.edu}

\address{$^\dagger$ Department of Mathematics, Columbia University, New York, NY 10027}

\email{phong@math.columbia.edu}

\address{$^\ddagger$ Department of Mathematics, Rutgers University, Piscataway, NJ 08854}

\email{jiansong@math.rutgers.edu}

\address{$^{\dagger\dagger}$ Department of Mathematics \& Computer Science, Rutgers University, Newark, NJ 07102}

\email{sturm@andromeda.rutgers.edu}

\begin{abstract}

{\footnotesize In this short note, we remove the small degeneracy assumption in our earlier works \cite{GPSS1, GPSS2}. This is achieved by a technical improvement of Corollary 5.1 in \cite{GPSS1}.  As a consequence, we establish the same geometric estimates for diameter, Green's functions and Sobolev inequalities  under an entropy bound for the K\"ahler metrics, without any small degeneracy assumption.  }

\end{abstract}

\maketitle

\section{Introduction} \label{secintro}

\setcounter{equation}{0}

 The classical works of Yau and his collaborators have built a wide range of geometric estimates for  Riemannian manifolds.  A lower bound for the Ricci curvature  is usually required to guarantee uniformity and compactness (c.f. \cite{ScYa, CL, CLY}). In our developing program to build geometric analysis on complex varieties with singularities \cite{GPSS1, GPSS2}, we managed to establish uniform diameter bounds, Sobolev inequalities and the spectral theorem for a large family of K\"ahler metrics on both smooth compact K\"ahler manifolds and normal K\"ahler varieties. Furthermore, these uniform estimates do not require any Ricci curvature bound. Instead they only depend on an entropy bound and a small degeneracy assumption (c.f. (\ref{degass})). 
 
 \smallskip

In his celebrated work \cite{Y} on the Calabi conjecture, Yau initiated the theory of global complex Monge-Amp\`ere equations.  The analytic theory was subsequently developed by Kolodziej \cite{K} in the framework of pluripotential theory. It was further extended to complex Monge-Amp\`ere equations on singular K\"ahler varieties \cite{EGZ, DPali}. Recently the PDE methods developed by Guo-Phong-Tong in \cite{GPT} gave an alternative proof of Kolodziej’s estimates. This approach has had many important geometric consequences \cite{GPa, GPSS1, GPSS2}.

We begin by reviewing the set-up and results of \cite{GPSS1, GPSS2}.  Let $(X, \theta_X)$ be an $n$-dimensional compact K\"ahler manifold equipped with a K\"ahler metric $\theta_X$. Let ${\mathcal{K}}(X)$ be the space of K\"ahler metrics on $X$ and let the $p$-th Nash-Yau entropy of a K\"ahler metric $\omega\in \mathcal{K}(X)$ associated to $(X, \theta_X)$ be defined by
\begin{equation}
\mathcal{N}_{X, \theta_X, p}(\omega) = \frac{1}{V_\omega}  \int_X \left| \log \left((V_\omega)^{-1}\frac{\omega^n}{\theta_X^n} \right)\right|^p \omega^n, ~ V_\omega= \int_X \omega^n = [\omega]^n, 
\end{equation}
for $p>0$. 

We introduce the following set of admissible metrics for given parameters $A,  K>0$, $p>n$,
\begin{equation} \label{admiss}
{\mathcal V}(X, \theta_X, n, A, p, K ) = \left\{ \omega\in {\mathcal{K}}(X):   [\omega]\cdot[\theta_X]^{n-1}\le A, ~  \mathcal{N}_{X, \theta_X, p}(\omega)  \le K   \right\}. 
\end{equation}
Let $\gamma$ be a non-negative continuous function. 
We further define a subset of ${\mathcal V}(X, \theta_X, n, A, p, K )$ by 
\begin{equation}\label{admiss2}
 {\mathcal  W} (X, \theta_X, n, A, p, K;   \gamma ) = \left\{ \omega \in {\mathcal{V}}(X, \theta_X, n, A, p, K): ~(V_\omega)^{-1}\frac{\omega^n} {\theta_X^n} \geq \gamma \right\}.   
\end{equation}

In the earlier works \cite{GPSS1, GPSS2} of the authors, uniform geometric estimates for the Green's function, Sobolev constant and diameter were established for K\"ahler metrics in ${\mathcal  W} (X, \theta_X, n, A, p, K;   \gamma )$  if $\gamma$ is a non-negative continuous function on $X$ satisfying 
\begin{equation}\label{degass}
\dim_{\mathcal{H}} \{\gamma =0\} < 2n-1.
\end{equation}
The small degeneracy assumption (\ref{degass}) requires the Monge-Amp\`ere measure to be uniformly positive away from a closed subset of $X$ of Hausdorff co-dimension strictly greater than $1$. In fact, this assumption naturally arises in most of geometric applications, because degeneration usually occurs along an analytic subvariety of $X$, which is closed and of complex codimension no less than $1$. The small degeneracy assumption was removed in \cite{GGZ} for the uniform diameter estimate if the underlying K\"ahler class lies in a compact set in the K\"ahler cone of a smooth K\"ahler manifold. Very recently, Vu \cite{Vu} applied the Sobolev inequality of \cite{GPSS2} to derive the diameter bounds for $\mathcal{V}(X, \theta_X, n, A, p, K)$. This leads us to try and remove the small degeneracy assumption in general, particularly for the Green's function and the Sobolev inequality in \cite{GPSS1, GPSS2}.

Indeed in this note, we will remove the small degeneracy assumption (\ref{degass}) for all the results of \cite{GPSS1, GPSS2, GPS}. This is achieved by the  following simple technical improvement of Corollary 5.1 of \cite{GPSS1}.

\begin{proposition} \label{mainprop}
Suppose $\omega\in \mathcal{V}(X, \theta_X, n, A, p, K)$.  If $v\in C^2$ satisfies $|\Delta_\omega v|\le 1$ and $\int_X v\omega^n = 0$, then there is a uniform constant $C=C(X, \theta_X, n, A, p, K)>0$ such that
$$\sup_X |v|\le C.$$ 

\end{proposition}

The proof of Proposition \ref{mainprop} is a straightforward iteration of the original argument in \cite{GPSS1}. It utilizes the same auxiliary complex Monge-Amp\`ere equation and is entirely based on the standard maximum principle. With Proposition \ref{mainprop}, all the results of \cite{GPSS1, GPSS2, GPS} will hold without the small degeneracy assumption (\ref{degass}). We will summarize these results in the general setting of normal K\"ahler spaces.

\begin{definition} \label{defak} Let $X$ be an $n$-dimensional compact normal K\"ahler space. Let $\pi: Y \rightarrow X$ be a log resolution of singularities and let $\theta_Y$ be a smooth K\"ahler metric  on the nonsingular model $Y$.  We define the set of admissible semi-K\"ahler currents 
$${\mathcal{AK}}(X, \theta_Y, n, p, A, K),$$  
 to be the set of any semi-K\"ahler current $\omega$ on $X$ satisfying the following conditions.

\begin{enumerate}

\item $[\omega]$ is a K\"ahler class on $X$ and $\omega$ has bounded local potentials. 

\smallskip

\item $[\pi^*\omega] \cdot [\theta_Y]^{n-1}\leq A$ and $[\omega]^n \geq A^{-1}$.  

\smallskip

\item The $p$-th Nash-Yau entropy is bounded for some $p>n$, i.e.
$${\mathcal{N}}_p (\omega) = \frac{1}{V_\omega} \int_Y \left|\log \frac{1}{V_\omega} \frac{(\pi^*\omega)^n}{(\theta_Y)^n} \right|^p (\pi^*\omega)^n  \leq K, $$
where $V_\omega= [\omega]^n$.

\medskip

\item The log volume measure ratio 
$$\log \left( \frac{(\pi^*\omega)^n}{V_\omega (\theta_Y)^n} \right)$$ has log type analytic singularities (c.f. Definition 7.2 of \cite{GPSS2}).

\end{enumerate}

\end{definition}

Obviously, if $X$ is a smooth K\"ahler manifold, 
$${\mathcal{AK}}(X, \theta_X, n, p, A, K) =\left\{ \omega\in  {\mathcal V}(X, \theta_X, n, A, p, K ): [\omega]^n \geq A^{-1} \right\}$$
considering the identity map $\pi=id: X \rightarrow X$. Proposition \ref{mainprop} enables us to enlarge the $\mathcal{AK}$-space in \cite{GPSS2}  to the $\mathcal{AK}$-space in Definition \ref{defak}.

Let $X$ be an $n$-dimensional normal K\"ahler variety. For any $p>n$ and any $\omega\in \mathcal{AK}(X, \theta_Y, n,p, A, K)$, we let $(\overline{X}, d, \omega^n)$ the metric measure space associated to $(X, \omega)$ as in \cite{GPSS2}. The Sobolev space $W^{1,2}(\overline{X},d, \omega^n)$ and its spectral theory are established   in \cite{GPSS2}  on  the metric measure space $(\overline{X}, d, \omega^n)$ associated to $\omega$.

We now state the geometric consequence of Proposition \ref{mainprop} for the works in \cite{GPSS1, GPSS2}. 
 
\begin{theorem} \label{mainthm} Let $X$ be an $n$-dimensional compact normal K\"ahler space. For any $\omega\in \mathcal{AK}(X, \theta_Y, n, p, A, K)$, the metric measure space $(\overline{X}, d, \omega^n)$ associated to $(X, \omega)$ satisfies the following properties. 

\begin{enumerate}

\item There exists $C=C(X, \theta_Y, n, p, A, K)>0$ such that  
$${\textnormal{diam}}(\overline{X}, d) \leq C.$$
In particular, $(\overline{X}, d)$ is a compact metric space.

\medskip 

\item  There exist $q>1$ and $C_S=C_S(X, \theta_Y, n, p,  A, K, q)>0$ such that 
$$
\Big(\int_{\overline{X}} | u  |^{2q}\omega^n   \Big)^{1/q}\le C_S \left( \int_{\overline{X}} |\nabla u|^2 ~\omega^n + \int_{\overline{X}} u^2 \omega^n \right) .
$$
for all $u\in W^{1, 2}(\overline{X}, d, \omega^n)$. 

\medskip

\item There exists $C=C(X, \theta_Y, n, p, A, K)>0$ such that the following trace formula holds for the heat kernel $H(x, y, t)$ of $(\overline{X}, d, \omega^n)$
$$H(x,x, t) \leq \frac{1}{V_\omega} + \frac{C}{V_\omega} t^{-\frac{q}{q-1}} $$
on $\overline{X}\times (0, \infty)$. 

\medskip

\item Let $0=\lambda_0 < \lambda_1 \leq \lambda_2 \leq ... $ be the increasing sequence of eigenvalues of the Laplacian $-\Delta_\omega$ on $(\overline{X}, d, \omega^n)$. Then there exists $c=c(X, \theta_Y, n, p, A, K)>0$ such that
$$\lambda_k \geq c k^{\frac{q-1}{q}}. $$

\end{enumerate}

\end{theorem}

We remark that Theorem \ref{mainthm} can be generalized to the collapsing cases if we drop the non-collapsing assumption $[\omega]^n\geq A^{-1}$ in Definition \ref{defak}. The normalization constant $I_\omega = [\pi^* \omega] \cdot [\theta_Y]^{n-1}$ will be added to the estimates analogous to those in Section 2 of \cite{GPSS2}. 
 



\section{Proof of Proposition \ref{mainprop}}

We will prove Proposition \ref{mainprop} in this section.  The following lemma is proved in \cite{GPSS1} (Corollary 5.1).  
\begin{lemma} \label{cor51}  
Suppose $\omega\in \mathcal{V}(X, \theta_X, n, A, p, K)$. If $v\in C^2(X)$ satisfies
\begin{equation}\label{meanv}
|\Delta_{\omega} v| \le 1\, \mbox{ and } \int_X v \omega^n = 0,
\end{equation}
then there is  $C=C(X, \theta_X, n, A, p, K)>0$ such that
$$\sup_X |v|\le C\left(1 + \frac{1}{[\omega]^n} \int_X |v| \omega^n  \right).$$

\end{lemma}

\begin{proposition}
Suppose $\omega\in \mathcal{V}(X, \theta_X, n, A, p, K)$.  If $v\in C^2(X)$ satisfies $|\Delta_\omega v|\le 1$ and $\int_X v\omega^n = 0$, then there is   $C=C(X, \theta_X, n, A, p, K)>0$ such that
$$\sup_X |v|\le C.$$ 

\end{proposition}
\begin{proof}
The proof is to repeat the argument of Lemma 5.1 of \cite{GPSS1} by the maximum principle. For convenience, we denote $C>0$  by uniform constants that only depend on $X$, $\theta_X$, $n$, $A$, $p$, $K$ throughout the argument. We first fix the constant $\delta = \delta(n)>0$  satisfying
$$\beta = \left(\delta+\frac{n^2}{n+1} \right) ^n \frac{(n+1)^n}{n^{2n} } =1.5.$$ 
Replacing $v$ by  $\delta v$, we may assume $|\Delta_\omega v| \le \delta$. Without loss of generality we may assume that 
$\alpha= \frac{1}{V_\omega}\int_{\{v> 0\}} \omega^n \le \frac{1}{2}$, otherwise we consider $-v$.

  Given $v_+ = \max(v, 0)$, we  consider the auxiliary complex Monge-Amp\`ere equation
\begin{equation} \label{ama1} 
(\omega+ \ddbar \psi)^n = \frac{v_+ } { B} \omega^n, \quad \sup_X \psi = 0,
\end{equation}
with $$B = \frac{1}{V_\omega} \int_X  v_+  \omega^n \in \mathbb{R}^+.$$ 
In fact, $\frac{1}{V_\omega} \int_X |v| \omega^n = 2B$ as $\int_X v \omega^n =0$. Therefore we can assume $B\geq 1$, otherwise, the proposition automatically holds. One can easily replace $v_+$ by a positive smoothing of $v_+$ as in \cite{GPSS1} and then take limits. For simplicity, we work with $v_+$ directly.     There exists $C>0$ such that 
$$\|v_+\|_{L^\infty(X)} \leq \|v\|_{L^\infty(X)} \leq CB$$
by applying Lemma \ref{cor51}. Therefore $\omega + \ddbar \psi \in \mathcal{V}(X, \theta_X, n, A, p, CK)$ for some uniform $C>0$. We can apply Corollary 4.1 in \cite{GPSS1} to derive the $L^\infty$-estimate 
\begin{equation}\label{psib}
\| \psi \|_{L^\infty(X)} \le C.
\end{equation}
for some uniform $C>0$.

We now consider the following auxiliary function constructed in  Lemma 5.1 of \cite{GPSS1}.
$$\Psi=  -\varepsilon (-\psi +  \varepsilon^{n+1}) ^{\frac{n}{n+1}} + ( v_+ ),~ \varepsilon  = (\beta B)^{1\over n+1}.$$
 We compute as in \cite{GPSS1}, at a maximum point $p$ of $\Psi$, so that 
\begin{align*}
0\ge &  \Delta_\omega \Psi\\
\ge &  \frac{n\varepsilon}{n+1} (-\psi +  \varepsilon^{n+1})^{-\frac{1}{n+1}} \Delta_\omega \psi + \Delta_\omega v_+\\
\ge & \frac{n^2\varepsilon}{n+1} (-\psi +  \varepsilon^{n+1})^{-\frac{1}{n+1}} \left(\frac{(\omega+ \ddbar\psi)^n}{\omega^n}\right)^{1/n} - \delta - \frac{n^2 \varepsilon }{n+1} (-\psi  + \varepsilon^{n+1})^{-\frac{1}{n+1}}\\
\ge  & \frac{n^2\varepsilon}{n+1} (-\psi +\varepsilon^{n+1})^{-\frac{1}{n+1}} \frac{(v_+  )^{1/n}}{B^{1/n}} - \delta - \frac{n^2}{n+1}. 
\end{align*}
Therefore at $p$, we have 
\begin{equation}\label{eqn:new2}v_+  \le \left(\delta + \frac{n^2}{n+1} \right) ^n \frac{(n+1)^n}{n^{2n} \varepsilon^n} B (-\psi + \varepsilon^{n+1}) ^{\frac{n}{n+1}} = \varepsilon  (-\psi +  \varepsilon^{n+1}) ^{\frac{n}{n+1}}.\end{equation}
This shows that at the maximum point of $\Psi$, $\Psi\le 0$, hence $\sup_X \Psi  \le 0$. Hence on $\Omega_+ = \{v> 0\}$, we have
$$v_+   \le \beta^{\frac{1}{n+1}} B^{{\frac{1}{n+1}}}  (-\psi  + \varepsilon^{n+1}) ^{\frac{n}{n+1}}\le \beta^{\frac{1}{n+1}} B^{{\frac{1}{n+1}}}  (C  + \beta B) ^{\frac{n}{n+1}},$$ 
by (\ref{psib}).
Integrating over $\Omega_+$, we have  
%
\begin{align*}
B = & \frac{1}{V_\omega}\int _{\Omega_+} v_+\omega^n 
\le  \beta^{\frac{1}{n+1}} B^{\frac{1}{n+1}} (C + \alpha \beta B )^{\frac{n}{n+1}} \alpha \le (\alpha \beta)^{\frac{1}{n+1}} B^{\frac{1}{n+1}} (C + \alpha \beta B )^{\frac{n}{n+1}}.  
\end{align*}
Therefore
\begin{equation}\label{eqn:new} B \le (\alpha \beta)^{\frac{1}{n} } (C + \alpha \beta  B ) \le (0.75 )^{\frac{1}{n} } (C + 0.75  B ) \leq C + 0.75 B 
\end{equation}
since $\alpha \beta \leq 0.75$. Immediately we have $B \le 4C$.

Finally, as observed earlier,  $\frac{1}{V_\omega}\int_{X} |v| \omega^n =2B \le 8C$.  We now can complete the proof of the proposition by invoking Lemma \ref{cor51} again.
\end{proof}

\section{Proof of Theorem \ref{mainthm}}

We will combine the results of \cite{GPSS1, GPSS2}  for $\mathcal{V}(X, \theta_X, n, p, A, K)$. Due to Proposition \ref{mainprop}, we can remove the small degeneracy assumption on a barrier function $\gamma$.
 
\begin{theorem} \label{smthm} Let $X$ be an $n$-dimensional compact K\"ahler manifold. For any $\omega\in \mathcal{V}(X, \theta_X, n, p, A, K)$, the following hold.

\begin{enumerate}

\item There exists $C=C(X, \theta_X, n, p, A, K)>0$ such that  
$${\textnormal{diam}}(X, d) \leq C.$$
In particular, $(X, d)$ is a compact metric space.

\medskip 

\item  There exist $q>1$ and $C_S=C_S(X, \theta_X, n, p,  A, K,  q)>0$ such that 
$$
\Big(\int_{X} | u  |^{2q}\omega^n   \Big)^{1/q}\le C_S \left( \int_{X} |\nabla u|_\omega^2 ~\omega^n + \int_{\overline{X}} u^2 \omega^n \right) .
$$
for all $u\in W^{1, 2}(X, d, \omega^n)$. 

\medskip

\item There exists $C=C(X, \theta_X, n, p, A, K)>0$ such that the following trace formula holds for the heat kernel $H(x, y, t)$ of $(X, d, \omega^n)$
$$H(x,x, t) \leq \frac{1}{V_\omega} + \frac{C}{V_\omega} t^{-\frac{q}{q-1}} $$
on $X\times (0, \infty)$. 

\medskip

\item Let $0=\lambda_0 < \lambda_1 \leq \lambda_2 \leq ... $ be the increasing sequence of eigenvalues of the Laplacian $-\Delta_\omega$ on $(X, d, \omega^n)$. Then there exists $c=c(X, \theta_X, n, p, A, K)>0$ such that
$$\lambda_k \geq c k^{\frac{q-1}{q}}. $$

\end{enumerate}

\end{theorem}

The proof of Theorem \ref{smthm} follows line by line in the argument of \cite{GPSS1, GPSS2} due to Proposition \ref{mainprop}. Now we can reduce Theorem \ref{mainthm} to Theorem \ref{smthm} by the smoothing arguments in Section 7 of \cite{GPSS2}. 

\bigskip

\noindent {\bf{Acknowledgements:}} The paper is inspired by the very recent work of Duc-Viet Vu \cite{Vu}. Upon completion of this paper, we were kindly informed by Vincent Guedj and Tat Dat T$\hat{\mathrm{o}}$ of their work \cite{GT} on the same improvement of $L^\infty$-estimate as in Proposition \ref{mainprop} using a different method. We very much appreciate their interest in our earlier works.

\bigskip


\end{document}